\documentclass[11p]{article}

\usepackage{times}
\usepackage{amsthm}
\usepackage{amsmath}
\usepackage{amssymb}
\usepackage{mathrsfs}
\usepackage{pb-diagram}
\usepackage{xcolor}
\usepackage{hyperref}

\def\te{\theta}

\def\N{\mathbb{N}}

\def\R{\mathbb{R}}

\def\U{\mathcal U}

\def\e{{\sf e}}

\def\r{{\rm r}}
\def\d{{\rm d}}
\def\bu{\bullet}

\def\({\left(}
\def\[{\left[}
\def\){\right)}
\def\]{\right]}

\def\G{{\sf G}}
\def\H{{\sf H}}

\def\p{\parallel}
\def\<{\langle}
\def\>{\rangle}
\providecommand{\norm}[1]{\lVert#1\rVert}

\usepackage{tikz-cd} 

 \newtheorem{thm}{Theorem}[section]
 \newtheorem{cor}[thm]{Corollary}
 
 \newtheorem{lem}[thm]{Lemma}
 \newtheorem{prop}[thm]{Proposition}
 \theoremstyle{definition}
 \newtheorem{defn}[thm]{Definition}
 \theoremstyle{remark}
 \newtheorem{rem}[thm]{Remark}
 \newtheorem{ex}[thm]{Example}
 \numberwithin{equation}{section}

\numberwithin{equation}{section}

\begin{document}


\title{A criterion for hypersymmetry on discrete groupoids}

\author{F. Flores}

\author{F. Flores
}

\maketitle


\begin{abstract}
Given a Fell bundle $\mathscr C\overset{q}{\to}\Xi$ over the discrete groupoid $\Xi$, we study the symmetry of the associated Hahn algebra $\ell^{\infty,1}(\Xi\!\mid\!\mathscr C)$ in terms of the isotropy subgroups of $\Xi$. We prove that $\Xi$ is symmetric (resp. hypersymmetric) if and only if all of the isotropy subgroups are symmetric (resp. hypersymmetric). We also characterize hypersymmetry using Fell bundles with constant fibers, showing that for discrete groupoids, 'hypersymmetry' equals 'rigid symmetry'.
\end{abstract}

\section{Introduction}\label{intro}

This article treats the symmetry of certain Banach $^*$-algebras connected with Fell bundles over discrete groupoids. The study of symmetry for groupoid algebras started not long ago, with Austad and Ortega \cite{AO}, and found a continuation in \cite{FJM}, where Jaur\'e, M\u antoiu and myself first treated the problem of symmetry for Fell bundles over discrete groupoids.

\begin{defn} \label{symmetric}
A Banach $^*$-algebra $\mathfrak B$ is called {\it symmetric} if the spectrum of $b^*b$ is positive for every $b\in\mathfrak B$ (this happens if and only if the spectrum of any self-adjoint element is real).
\end{defn}

In this regard, the main result of the paper is the following. 

\begin{thm}\label{promised}
Let $\Xi$ be a discrete groupoid. Then: \begin{enumerate}
    \item The algebra $\ell^{\infty,1}(\Xi)$ is symmetric if and only if every isotropy group is symmetric.
    \item The algebra $\ell^{\infty,1}(\Xi\!\mid\!\mathscr C)$ is symmetric for every Fell bundle $\mathscr C$ over $\Xi$ if and only if every isotropy group is hypersymmetric.
\end{enumerate}
\end{thm}

This result is interesting as it reduces the question of the (hyper)symmetry of a given groupoid to the (hyper)symmetry of its isotropy subgroups and the study of group $\ell^1$-algebras is far more developed. For example, see \cite{AO,FJM,JM,Ma,Po,LP,Ku,LeN} and the reference therein.

The article is divided in 3 parts: Section \ref{pre} deals with preliminares; there we introduce Fell bundles, the Hahn algebra of a Fell bundle and define what (hyper)symmetry for a groupoid means. In Section \ref{char} we introduce a characterization for hypersymmetry using only Fell bundles with constant fibers. Its analogous to the characterization made for groups by Jaur\'e and M\u antoiu in \cite{JM}. Finally, Section \ref{rigsym} deals with the proof of Theorem \ref{promised}. This is achieved by writing a general discrete groupoid as a disjoint union of transitive groupoids and proving that transitive groupoids are isomorphic to transitive transformation groups. Using the available theory for groups yields the desired result.

\section{Preliminaries}\label{pre}

Let $\Xi$ be a groupoid, with unit space $\U:=\Xi^{(0)}$\,, source map $\d$ and range map ${\rm r}$\,. The $\d$ and $\r$-fibers are $\Xi_u=\{\xi\in\Xi\mid \d(\xi)=u\}$ and $\Xi^u=\{\xi\in\Xi\mid \r(\xi)=u\}$. The set of composable pairs is 
$$\Xi^{(2)}\!:=\{(x,y)\!\mid\!{\rm r}(y)=\d(x)\}\,.
$$ 
The isotropy group associated with the unit $u\in \U$ is $\Xi_u^u=\Xi_u\cap \Xi^u$, it is a subgroupoid of $\Xi$, which happens to be a group. We endow the groupoid $\Xi$ with the discrete topology. Let us introduce an important class of groupoids which will be useful.

\begin{defn}\label{transgroup}
A groupoid $\Xi$ is called transitive if for any pair $u,v\in \U$, there exists $x\in\Xi$ such that $\r(x)=u$ and $\d(x)=v$.
\end{defn}

The concept of a transitivity is borrowed from the theory of dynamical systems. And it comes from a natural but hidden action of the groupoid on its unit space. So the groupoid is transitive if and only if this action is transitive, cf \cite[Example 2.2, Corollary 7.8]{FM}. 

In this article we are going to work with {\it Fell bundles} $\mathscr C\overset{q}{\to}\Xi$ over the groupoid ${\Xi}$ (see \cite {Kum,MW}). A Fell bundle is composed of fibers and it satisfies that each fiber $\mathfrak C _x\!:=q^{-1}(\{x\})$ is a Banach space with norm $\p\!\cdot\!\p_{\mathfrak C_x}$\,, the topology of $\mathscr C$ coincides with the norm topology on each fiber, there are antilinear continuous involutions 
$$
\mathfrak C _x\ni \!a\to a^{\bu}\!\in\mathfrak C _{x^{-1}}
$$ 
and for all $(x,y)\in\Xi^{(2)}$ there are continuous multiplications 
$$
\mathfrak C _x\!\times\!\mathfrak C _y\ni(a,b)\to a\bu b\in\mathfrak C _{xy}
$$ 
satisfying the following axioms valid for $a\in\mathfrak C_x\,,b\in\mathfrak C_y\,$ and $(x,y)\in\Xi^{(2)}$\,: 
\begin{itemize}
\item $\p\!ab\!\p_{\mathfrak C_{xy}}\,\le\,\p\!a\!\p_{\mathfrak C_{x}}\p\!b\!\p_{\mathfrak C_{y}}$\,,
\item $(ab)^\bu=b^\bu a^\bu$,
\item $\p\!a^\bu a\!\p_{\mathfrak C_{\d(x)}}=\,\p\!a\!\p_{\mathfrak C_{x}}^2$\,,
 \item $a^\bu a$ is positive in $\mathfrak C_{\d(x)}$\,.
 \end{itemize}
From these axioms it follows that $\mathfrak C _x$ is a $C^*$-algebra for every unit $x\in \U$\,. Sometimes we simply write $\mathscr C=\bigsqcup_{x\in\Xi}\mathfrak C _x$ for the Fell bundle.

Our object of study is {\it the Hahn algebra} $\ell^{\infty,1}(\Xi\!\mid\!\mathscr C)$ adapted to Fell bundles  \cite{MW}, which in our case it is formed by the sections $\Phi:\Xi\to\mathfrak C$ (thus satisfying $\Phi(x)\in\mathfrak C_x$ for every $x\in\Xi$) that can be obtained as a limit of finitely-supported sections in {\it the Hahn-type norm} 
\begin{equation}\label{parts}
\norm{\Phi}_{\infty,1}\,:=\max\Big\{\sup_{u\in \U}\sum_{{\rm r}(x)=u}\!\p\!\Phi(x)\!\p_{\mathfrak C_x}\,,\,\sup_{u\in \U}\sum_{\d(x)=u}\!\p\!\Phi(x)\!\p_{\mathfrak C_x}\!\Big\}.
\end{equation}
 It is a Banach $^*$-algebra under the multiplication
\begin{equation}\label{tiplication}
(\Phi* \Psi)(x):=\sum_{yz=x}\Phi(y)\bullet\Psi\big(z)
\end{equation}
and the involution
\begin{equation}\label{tion}
\Phi^*(x):=\Phi(x^{-1})^\bu.
\end{equation}

\begin{rem}\label{convergence}
Let us point out some of the nature of the functions in $\ell^{\infty,1}(\Xi\!\mid\!\mathscr C)$. If $\Phi_n\in \ell^{\infty,1}(\Xi\!\mid\!\mathscr C)$ is a sequence of sections with finite support and $\Phi_n\to \Phi$, then the convergence is uniform. Indeed, let $x\in\Xi$ and observe that \begin{align*}
    \norm{\Phi_n(x)-\Phi(x)}_{\mathfrak C_x}&\leq \sum_{y\in\Xi^{r(x)}}\norm{\Phi_n(y)-\Phi(y)}_{\mathfrak C_y} \\
    &\leq \sup_{u\in\U}\sum_{y\in\Xi^{u}}\norm{\Phi_n(y)-\Phi(y)}_{\mathfrak C_y} \\
    &\leq\norm{\Phi_n-\Phi}_{\ell^{\infty,1}(\Xi\,\mid\,\mathscr C)}.
\end{align*} So $\lim_n\sup_{x\in\Xi}\norm{\Phi_n(x)-\Phi(x)}_{\mathfrak C_x}=0$. This implies that the function $\Phi$ has countable support and vanishes at $\infty$.
\end{rem}

We denote by $C^*(\Xi\,\vert\,\mathscr C)$ the enveloping $C^*$-algebra of the Hahn algebra $\ell^{\infty,1}(\Xi\!\mid\!\mathscr C)$. It is a known fact that $\ell^{\infty,1}(\Xi\!\mid\!\mathscr C)$ is a dense $^*$-subalgebra of $C^*(\Xi\!\mid\!\mathscr C)$.

\begin{defn} \label{groupsymmetric}
\begin{enumerate}
\item[(i)]
The discrete groupoid $\Xi$ is called {\it symmetric} if the convolution Banach $^*$-algebra $\ell^{\infty,1}(\Xi)$ is symmetric.
\item[(ii)]
The discrete groupoid $\Xi$ is called {\it hypersymmetric} if given any Fell bundle $\mathscr C\!=\bigsqcup_{x\in\Xi}\mathfrak C _x$\,, the Banach $^*$-algebra $\ell^{\infty,1}(\Xi\!\mid\!\mathscr C)$ is symmetric.
\end{enumerate}
\end{defn}

\begin{ex}
If $\Pi\subset X\!\times\!X$ is an equivalence relation on $X$, one can make $\Xi=\Pi$ a discrete groupoid by defining the operations 
\begin{equation*}
\begin{split}
\d(x,y)=(y,y)\,,\quad &\r(x,y)=(x,x)\,,\quad (x,y)(y,z)=(x,z)\,,\quad(x,y)^{-1}\!=(y,x)\,.   
\end{split}
\end{equation*} 
The unit space is $\,\mathcal U={\sf Diag}(X)$ and it gets canonically identified with $X$, via the homeomorphism $(x,x)\mapsto x$\,. In this case all of the isotropy groups correspond to the trivial group $\Pi_u^u=\{(u,u)\}$, so Theorem \ref{full} will guarantee that $\Pi$ is hypersymmetric. A particular examples is the so called {\it pair groupoid}, $\Pi=X\times X$.
\end{ex}

\section{A characterization of hypersymmetry for discrete groupoids}\label{char}

In \cite{FJM} we introduced some special Fell bundles arising from Hilbert bundles to characterize the hypersymmetry of a discrete groupoid. However, in this paper we will improve this characterization: One can actually verify hypersymmetry by looking at much simpler algebras, associated to Fell bundles with constant fibers. Let us make precise the Fell bundles of interest.

\begin{defn}\label{groupoidaktion}
By {\it a (left) groupoid action} of a discrete groupoid $\Xi$ on the $C^*$-bundle $\mathscr A\!:=\bigsqcup_{u\in \U}\mathfrak A_u\overset{p}{\to}\U$ over its unit space  we understand
a continuous map 
$$
\mathscr A\rtimes\Xi:=\{(\alpha,x)\in\mathscr A\!\times\!\Xi\!\mid\!p(\alpha)=\d(x)\}\ni(\alpha,x)\to \mathcal T(\alpha,x)\equiv\mathcal T_x(\alpha)\in\mathscr A,
$$ 
satifying the axioms
\begin{enumerate}
\item[(a)] $p\big[\mathcal T_x(\alpha)\big]={\rm r}(x)$\,, $\forall\,x\in\Xi\,,\,\alpha\in\mathfrak A_{{\rm d}(x)}$\,,
\item[(b)] each $\mathcal T_x$ is a $^*$-isomorphism $:\mathfrak A_{{\rm d}(x)}\!\to\mathfrak A_{{\rm r}(x)}$\,,
\item[(c)] $\mathcal T_u={\rm id}_{\mathfrak A_u}$\,, $\forall\,u\in \U$\,,
\item[(d)] if $(x,y)\in\Xi^{(2)}$ and $(\alpha,y)\in\mathscr A\rtimes\Xi$\,, then $\big(\mathcal T_y(\alpha),x\big)\in\mathscr A\rtimes\Xi$ and $\mathcal T_{xy}(\alpha)=\mathcal T_x\big[\mathcal T_y(\alpha)\big]$\,.
\end{enumerate}
\end{defn}

\begin{defn}\label{aktionbundle}
Let $\mathcal T$ be a groupoid action of $\Xi$ on the $C^*$-bundle $\mathscr A\!:=\bigsqcup_{u\in \U}\mathfrak A_u\overset{p}{\to}\U$. We define its associated Fell bundle as follows. The underlying space is $\mathscr C_{\mathcal T}:=\mathscr A\rtimes\Xi$ with the topology inherited from the product topology and the obvious projection $q$. We endow it with the operations 
$$
(\alpha,x)\bu(\beta,y):=(\alpha\mathcal T_x(\beta),xy)\,, \textup{when } (x,y)\in\Xi^{(2)}
$$ and
$$
(\alpha,x)^\bu:=(\mathcal T_{x^{-1}}(\alpha^*),x^{-1})
$$
to get a Fell bundle over $\Xi$\,. A section is now a map $\Phi:\Xi\to\mathscr C_{\mathcal T}$ such that 
$$
\Phi(x)\equiv\big(\varphi(x),x\big)\in\mathfrak C_x=\mathfrak A_{{\r}(x)}\!\times\!\{x\}\,,\quad\forall\,x\in\Xi\,.
$$
So we may identify every section $\Phi\in \ell^{\infty,1}(\Xi\,\vert\,\mathscr C_{\mathcal T})$ with a function $\Phi:\Xi\to\mathscr A$ (note the abuse of notation), such that $ \Phi(x)\in\mathfrak A_{{\r}(x)}$. 

We will also denote the algebra $\ell^{\infty,1}(\Xi\,\vert\,\mathscr C_{\mathcal T})$ by $\ell^{\infty,1}_{\mathcal T}(\Xi,\mathscr A)$ to recall its particular nature. If the action $\mathcal T$ is trivial, meaning that $\mathfrak A_u\equiv \mathfrak A$ for all $u\in\U$ and $\mathcal T_x\equiv{\rm id}_{\mathfrak A}$ for all $x\in\Xi$, then we denote the resulting algebra simply by $\ell^{\infty,1}(\Xi,\mathfrak A)$.
\end{defn}

\begin{rem}
Let $\Phi,\Psi\in\ell^{\infty,1}_{\mathcal T}(\Xi,\mathscr A)$. In this case, one may write the algebraic laws as
\begin{equation}\label{unaca}
\big[\Phi*\Psi\big](x)=\!\sum_{y\in\Xi^{{\rm r}(x)}}\!\Phi(y)\mathcal T_y\big[\Psi(y^{-1}x)\big]\,,
\end{equation}
\begin{equation}\label{binaca}
\Phi^*(x)=\mathcal T_{x}\big[\Phi(x^{-1})\big]^*.
\end{equation}
And the Hahn-type norm as
\begin{equation}\label{trinaca}
\p\!\Phi\!\p_{\ell^{\infty,1}_{\mathcal T}(\Xi,\mathscr A)}=\max\Big\{\sup_{u\in \U}\sum_{{\rm r}(x)=u}\!\p\!\Phi(x)\!\p_{\mathfrak A_{\r(x)}}\,,\,\sup_{u\in \U}\sum_{{\rm d}(x)=u}\!\p\!\Phi(x)\!\p_{\mathfrak A_{\r(x)}}\!\Big\}\,.
\end{equation}
\end{rem}

\begin{lem}\label{embeddinglemma}
Let $\mathscr C\!=\bigsqcup_{x\in\Xi}\mathfrak C_x$ be a Fell bundle over the discrete groupoid $\Xi$. Then there exists and an isometric $^*$-monomorphism \begin{equation}\label{eq}
    \varphi:\ell^{\infty,1}(\Xi\,\vert\,\mathscr C)\to \ell^{\infty,1}\big(\Xi,C^*(\Xi\,\vert\,\mathscr C)\big).
\end{equation}
\end{lem}

\begin{proof}
Set $\mathfrak A:=C^*(\Xi\,\vert\,\mathscr C)$ and, for every $x\in \Xi$ we embed $\mathfrak C_x$ into $\ell^{\infty,1}(\Xi\,\vert\,\mathscr C)\subset\mathfrak A$, by setting for each $a\in\mathfrak C_x$
$$
(\te_x a)(y):=a\ \ {\rm if}\ \ y=x\,,\quad(\te_x a)(y):=0_{\mathfrak C_x}\ \ {\rm if}\ \ y\ne x\,.
$$
It is not hard to prove that if $(x,y)\in\Xi^{(2)}$, then $$\te_xa*\te_yb=\te_{xy}(a\bu b)\quad\textup{ and }\quad (\te_xa)^*=\te_{x^{-1}}a^\bu$$ hold. On the other hand, one also has $\norm{\te_xa}_{\mathfrak A}=\norm{a}_{\mathfrak C_x}$: If $x\in \U$, this equality holds because $\te_x:\mathfrak C_x\to \mathfrak A$ is a $^*$-monomorphism of $C^*$-algebras, hence isometric. If $x\in\Xi$ is not an unit, then one may apply the -now standard- trick 
$$
\norm{\te_x a}^2_{\mathfrak A}=\norm{(\te_x a)^**(\te_x a)}_{\mathfrak A}=\norm{\te_{x^{-1}x}(a^\bu\bu a)}_{\mathfrak A}=\norm{a^\bu\bu a}_{\mathfrak C_{x^{-1}x}}=\norm{a}^2_{\mathfrak C_{x}}
$$ 
to conclude that $\te_x$ preserves the mentioned norms. This allows us to successfully define
$$
\varphi(\Phi)(x)=\te_x \Phi(x), \textup{  for }\Phi\in \ell^{\infty,1}(\Xi\,\vert\,\mathscr C)
$$
and have an isometry:
\begin{align*}
    \norm{\varphi(\Phi)}_{\ell^{\infty,1}(\Xi,{\mathfrak A})}&=\max\Big\{\sup_{u\in \U}\sum_{{\rm r}(x)=u}\!\p\!\te_x \Phi(x)\!\p_{\mathfrak A}\,,\,\sup_{u\in \U}\sum_{\d(x)=u}\!\p\!\te_x \Phi(x)\!\p_{\mathfrak A}\!\Big\} \\
    &=\max\Big\{\sup_{u\in \U}\sum_{{\rm r}(x)=u}\!\p\! \Phi(x)\!\p_{\mathfrak C_x}\,,\,\sup_{u\in \U}\sum_{\d(x)=u}\!\p\!\Phi(x)\!\p_{\mathfrak C_x}\!\Big\} \\
    &=\norm{\Phi}_{\ell^{\infty,1}(\Xi\,\vert\,\mathscr C)}.
\end{align*} Now we check that $\varphi$ is an $^*$-homomorphism, \begin{align*}
    \big[\varphi(\Phi)*\varphi(\Psi)\big](x)&=\sum_{yz=x} \te_y \Phi(y)*\te_z\Psi(z) \\
    &=\te_x \sum_{yz=x} \Phi(y)\bu\Psi(z) \\
    &=\varphi(\Phi*\Psi)(x)
\end{align*} and 
$$
    \varphi(\Phi^*)(x)=\te_x \Phi^*(x)=\te_x \Phi(x^{-1})^\bu=\big[\te_{x^{-1}}\Phi(x^{-1})\big]^*=\varphi(\Phi)(x^{-1})^*=\varphi(\Phi)^*(x).
$$ This finishes the proof.\end{proof}

\begin{rem}\label{tensorprodukt}
Observe that $\ell^{\infty,1}(\Xi,\mathfrak A)\cong \ell^{\infty,1}(\Xi)\,\hat{\otimes}\,\mathfrak A$, where $\hat{\otimes}$ denotes the projective tensor product. Indeed, given $(\varphi,a)\in \ell^{\infty,1}(\Xi)\times\mathfrak A$, define the function $\varphi\otimes a$ by $$\big[\varphi\otimes a\big](x):=a\varphi(x), \quad\forall x\in\Xi.$$ The map $(\varphi,a)\mapsto \varphi\otimes a$ defined in $\ell^{\infty,1}(\Xi)\times\mathfrak A\to \ell^{\infty,1}(\Xi,\mathfrak A)$ is bilinear, has norm $1$ (it satisfies $\norm{\varphi\otimes a}_{\ell^{\infty,1}(\Xi,\mathfrak A)}\leq \norm{a}_{\mathfrak A}\norm{\varphi}_{\ell^{\infty,1}(\Xi)}$) and it extends to $^*$-isomorphism. 
$$
\iota: \ell^{\infty,1}(\Xi)\,\hat{\otimes}\,\mathfrak A\to \ell^{\infty,1}(\Xi,\mathfrak A)
$$ 
\end{rem}

The following corollary effectively reduces our concerns to the study of tensor products, just as in the group case (cf. \cite[Theorem 2.4]{JM}).

\begin{cor}\label{usecor}
A discrete groupoid $\Xi$ is hypersymmetric if and only if the Banach $^*$-algebra $\ell^{\infty,1}(\Xi,\mathfrak A)\cong \ell^{\infty,1}(\Xi)\,\hat{\otimes}\,\mathfrak A$ is symmetric, for every $C^*$-algebra $\mathfrak A$.
\end{cor}
\begin{proof}
Every algebra of the form $\ell^{\infty,1}(\Xi,\mathfrak A)$ comes from a Fell bundle (recall Definition \ref{aktionbundle}), so it is symmetric if $\Xi$ is hypersymmetric. On the other hand, because of Lemma \ref{embeddinglemma}, given any Fell bundle $\mathscr C$, $\ell^{\infty,1}(\Xi\,\vert\,\mathscr C)$ may be identified as a closed $^*$-subalgebra of some algebra of the form $\ell^{\infty,1}(\Xi,\mathfrak A)$. So it will be symmetric if the latter is symmetric, by \cite[Theorem 11.4.2]{Pa2}.
\end{proof}

\begin{rem}\label{rigid}
In the group case, this condition has been called 'rigid symmetry' by many authors (including myself) and it was introduced by Leptin and Poguntke in \cite{LP}. It can also be seen in \cite{AO,FJM,JM,Ma,Po}.
\end{rem}

Before going into the following sections, let us simplify some notation. Let $\mathfrak A,\mathfrak B$ be Banach $^*$-algebras. We will denote by $\mathfrak A\hookrightarrow\mathfrak B$ the fact that there exists an isometric $^*$-monomorphism $\iota:\mathfrak A\to\mathfrak B$. In this language, the conclusion of Lemma \ref{embeddinglemma} can be written as $\ell^{\infty,1}(\Xi\,\vert\,\mathscr C)\hookrightarrow \ell^{\infty,1}\big(\Xi,C^*(\Xi\,\vert\,\mathscr C)\big)$. On the other hand, the dual notation $\mathfrak A\twoheadrightarrow\mathfrak B$ means that there exists some contractive $^*$-epimorphism $\pi:\mathfrak A\to\mathfrak B$.

\section{The result for discrete groupoids}\label{rigsym}

The rest of the paper is devoted to prove Theorem \ref{full}. We will follow the strategy detailed in the introduction, starting with a complete (well-known) characterization of discrete transitive groupoids as group transformation groupoids.

\begin{defn}\label{starlet}
Let $\gamma$ be a continuous action of the (discrete) group $\G$ on the topological space $X$, we define {\it  the transformation groupoid} $\Xi:=\G\ltimes_\gamma\!X$ associated to it as follows. As a topological space $\G\ltimes_\gamma\!X$, it is just $\G\times X$, the maps $\r,\d$ are given by $\d(a,x)=x$ and $\r(a,x)=\gamma_a(x)$. The composition is $\big(b,\gamma_a(x)\big)(a,x):=(ba,x)$ and inversion reads $(a,x)^{-1}:=\big(a^{-1}\!,\gamma_a(x)\big)$\,. The unit space is $\U=\{{\e}\}\times X\equiv X$.

\end{defn}

\begin{prop}\label{lema}
Let $\Xi$ be a discrete transitive groupoid with isotropy group $\G$. There is an abelian group structure on $\U$ and an action $\gamma$ of $\G'=\G\times \U$ on $\U$ such that $\Xi\cong \G'\ltimes_\gamma \U$.
\end{prop}
\begin{proof}
Let us give $\U$ some abelian group structure with additive notation and define the action of $\G'$ by $\gamma_{(\alpha,w)}(v)=w+v$. In order to construct an isomorphism $\varphi$, fix some $u\in\U$ and realize $\G$ as $\Xi_u^u$. Since $\Xi$ is transitive, for every $v\in \U$, there exists an arrow $z_v\in\Xi$ such that $\d(z_v)=v$ and $\r(z_v)=u$. Then $$\varphi: \G'\ltimes_\gamma \U\to \Xi \textup{ defined by }\varphi\big((x,w),v\big)=z^{-1}_{w+v}x z_{v}, \textup{ is the required isomorphism.}$$
Let us verify that $\varphi$ is indeed a groupoid isomorphism: \begin{itemize}
    \item[(i)] If $\big((x_1,w_1),w_2+v\big)\big((x_2,w_2),v\big)=\big((x_1x_2,w_1+w_2),v\big)$, then 
    \begin{align*}
    \varphi\big((x_1x_2,w_1+w_2),v\big)&=z^{-1}_{w_1+w_2+v}x_1x_2z_{v} \\
    &=z^{-1}_{w_1+w_2+v}x_1z_{w_2+v}z_{w_2+v}^{-1}x_2z_{v} \\
    &=\varphi\big((x_1,w_1),w_2+v\big)\varphi\big((x_2,w_2),v\big)
\end{align*}
    \item[(ii)] If $\big((x,w),v\big)\in \G'\ltimes_\gamma \U$, then $\big((x,w),v\big)^{-1}=\big((x^{-1},-w),w+v\big)$ and
    \begin{align*}
    \varphi\big((x^{-1},-w),w+v\big)&=z^{-1}_v x^{-1}z_{w+v} \\
    &=(z^{-1}_{w+v}x z_{v})^{-1}=\varphi\big((x,w),v\big)^{-1}.
\end{align*}
    \item[(iii)] The inverse function $\varphi^{-1}$ is given by $\varphi^{-1}(\xi)=\big((z_{\r(\xi)}\xi z_{\d(\xi)}^{-1},\r(\xi)-\d(\xi)),\d(\xi)\big)$.
\end{itemize}

\end{proof}

\begin{rem}
It follows from Proposition \ref{lema} that the pair groupoid $\U\times \U$ is isomorphic to $\U\ltimes_{\gamma|_\U} \U$. (Here $\G=\{\e\}$ is the trivial group)
\end{rem}

While the commutativity of $\U$ was not essential for the previous proof -any group structure would have worked out-, it will be of vital importance in the future (cf. the proof of Theorem \ref{trans}). That is why it will be occasionally remarked in the following propositions.

The following lemma requires some notation. Given a $C^*$-algebra $\mathfrak A$ and a group action $\gamma:\G\to {\rm Bij}(\U)$, denote by $\Gamma$ the $\G$-action on $\mathcal C_0(\U,\mathfrak A)$ satisfying $\Gamma_x(f)(u)=f\big(\gamma_{x^{-1}}(u)\big)$. 

\begin{lem}\label{reduction}
Suppose that the discrete group $\G$ acts on $\U$ via $\gamma$ and $\mathfrak A$ is a $C^*$-algebra. Then
\begin{equation}\label{eq1}
    \ell^1_\Gamma\big(\G,\mathcal C_0(\U,\mathfrak A)\big)\twoheadrightarrow \ell^{\infty,  1}(\G\ltimes_\gamma \U,\mathfrak A).
\end{equation}
\end{lem}

\begin{proof}
Define $\pi:\ell^1_\Gamma\big(\G,\mathcal C_0(\U,\mathfrak A)\big)\to \ell^{\infty,  1}(\G\ltimes_\gamma \U,\mathfrak A)$ by the formula $\pi(\Phi)(x,u)=\Phi(x)\big(\gamma_{x}(u)\big)$. This map is well-defined and clearly surjective, while it also satisfies
\begin{align*}
    \norm{\pi(\Phi)}_{\ell^{\infty,  1}(\G\ltimes_\gamma \U,\mathfrak A)}&=\max\Big\{\sup_{u\in \U}\sum_{x\in \G}\!\p\! \Phi(x)\big(\gamma_{x}(u)\big)\!\p_{\mathfrak A}\,,\,\sup_{u\in \U}\sum_{x\in\G}\!\p\! \Phi(x)(u)\!\p_{\mathfrak A}\!\Big\} \\ 
    &\leq \sum_{x\in\G}\sup_{u\in \U}\!\p\! \Phi(x)(u)\!\p_{\mathfrak A}=\sum_{x\in\G}\norm{\Phi(x)}_{\mathcal C_0(\U,\mathfrak A)}=\norm{\Phi}_{\ell^1_\Gamma(\G,\mathcal C_0(\U,\mathfrak A))}
\end{align*} and \begin{align*}
    \pi(\Phi*\Psi)(x,u)&=\big[\Phi*\Psi\big](x)\big(\gamma_{x}(u)\big) \\
    &=\Big[\sum_{y\in \G}\Phi(y)\Gamma_y\big[\Psi(y^{-1}x)\big]\Big]\big(\gamma_{x}(u)\big) \\
    &=\sum_{y\in \G}\Phi(y)\big(\gamma_{x}(u)\big)\Psi(y^{-1}x)\big(\gamma_{y^{-1}x}(u)\big) \\
    &=\sum_{y\in \G}\pi(\Phi)(y,\gamma_{y^{-1}x}(u))\pi(\Psi)(y^{-1}x,u) \\
    &=\big[\pi(\Phi)*\pi(\Psi)\big](x,u).
\end{align*} Finally, we see that \begin{align*}
    \pi(\Phi^*)(x,u)&=\Phi^*(x)\big(\gamma_{x}(u)\big) \\
    &=\Gamma_x\big[\Phi(x^{-1})^*]\big(\gamma_{x}(u)\big)=\Phi(x^{-1})(u)^*=\big[\pi(\Phi)(x^{-1},\gamma_{x}(u))\big]^*=\pi(\Phi)^*(x,u).
\end{align*} Hence $\pi$ is a contractive $^*$-epimorphism.
\end{proof}

Lemma \ref{reduction} will be used in the following form, which allows us to focus on the study of $\ell^1$-algebras arising from group $C^*$-dynamical systems.

\begin{cor}\label{redstat}
If $\ell^1_\Gamma\big(\G,\mathcal C_0(\U,\mathfrak A)\big)$ is a symmetric Banach $^*$-algebra, then $\ell^{\infty,  1}(\G\ltimes_\gamma \U,\mathfrak A)$ is also symmetric.
\end{cor}

\begin{proof}
In Lemma \ref{reduction}, we showed that $\ell^1_\Gamma\big(\G,\mathcal C_0(\U,\mathfrak A)\big)\twoheadrightarrow \ell^{\infty,  1}(\G\ltimes_\gamma \U,\mathfrak A)$, so the conclusion follows from \cite[Theorem 11.4.2]{Pa2}.
\end{proof}

\begin{prop}\label{separation}
Let $(\G\times\H, \Gamma, \mathfrak A)$ be a $C^*$-dynamical system and assume that $\G$ acts trivially on $\mathfrak A$. Then one has
\begin{equation}\label{eq2}
    \ell^1_\Gamma(\G',\mathfrak A)\cong \ell^1(\G)\,\hat{\otimes}\,\ell^1_{\Gamma}(\H,\mathfrak A),
\end{equation}
where $\G':=\G\times\H$.
\end{prop}

\begin{proof}
Observe that $$\iota: \ell^1_\Gamma(\G',\mathfrak A)\to\ell^1\big(\G,\ell^1_{\Gamma}(\H,\mathfrak A)\big)\quad\textup{ defined by }\quad\iota(\Phi)(x)(y)=\Phi(x,y) 
$$ is an isometric $^*$-isomorphism. Indeed, bijectivity is direct, while
\begin{align*}
    \norm{\iota(\Phi)}_{\ell^1(\G,\ell^1_{\Gamma}(\H,\mathfrak A))}&=\sum_{x\in\G}\norm{\iota(\Phi)(x)}_{\ell^1_{\Gamma}(\H,\mathfrak A)}=\sum_{x\in\G}\sum_{y\in\H}\norm{\Phi(x,y)}_{\mathfrak A}=\norm{\Phi}_{\ell^1_\Gamma(\G',\mathfrak A)}
\end{align*} and, identifying $(1_\G,b)\in\G'$ with $b\in\H$, \begin{align*}
    \big[\iota(\Phi)*\iota(\Psi)\big](x)(y)&=\Big[ \sum_{a\in \G} \iota(\Phi)(a)*\iota(\Psi)(a^{-1}x)\Big](y) \\
    &=\sum_{a\in \G}\sum_{b\in \H} \iota(\Phi)(a)(b)\Gamma_{b}\big[\iota(\Psi)(a^{-1}x)(b^{-1}y)\big] \\
    &=\sum_{(a,b)\in \G\times\H}\Phi(a,b){\Gamma}_{(a,b)}\big[\Psi(a^{-1}x,b^{-1}y)\big] \\
    &=\iota(\Phi*\Psi)(x)(y). 
\end{align*} Finally, \begin{align*}
\iota(\Phi^*)(x)(y)=\Phi^*(x,y)=\Gamma_{y}\big[\Phi(x^{-1},y^{-1})^*\big]=\iota(\Phi)(x^{-1})^*(y)=\iota(\Phi)^*(x)(y).
\end{align*} So $\ell^1_\Gamma(\G',\mathfrak A)\cong \ell^1\big(\G,\ell^1_{\Gamma}(\H,\mathfrak A)\big)\cong \ell^1(\G)\,\hat{\otimes}\,\ell^1_{\Gamma}(\H,\mathfrak A)$.
\end{proof}

We can put together the previous lemmas to obtain the following result, which is Theorem \ref{promised} for transitive groupoids.

\begin{thm}\label{trans}
Let $\Xi$ be a discrete transitive groupoid with isotropy subgroup $\G$. If $\G$ is symmetric (resp. hypersymmetric), then $\Xi$ is symmetric (resp. hypersymmetric).
\end{thm}

\begin{proof}
Let us divide the proof in two cases, the first one being about hypersymmetry. Because of Proposition \ref{lema}, we may assume that $\Xi=\G'\ltimes_\gamma \H$, with $\G'=\G\times \H$ and $\gamma_{(g,h)}(k)=h+k$.

First suppose that $\G$ is hypersymmetric. Because of corollaries \ref{usecor} and \ref{redstat}, it is enough to show that $\G\times \H$ is hypersymmetric (or 'rigidly symmetric', see Remark \ref{rigid}). But this follows from the fact that $\H$ is abelian and \cite[Theorem 7]{LP}.

Now suppose that $\G$ is only symmetric. Note that $\gamma|_\H$ coincides with the action of $\H$ on itself by left translation, so we will denote it by ${\rm lt}$. Then
    \begin{align*}\label{incl2}
    \ell^1_\Gamma\big(\G',\mathcal C_0(\H)\big)&\overset{\eqref{eq2}}{\cong} \ell^1(\G)\,\hat{\otimes}\,\ell^1_{\rm lt}\big(\H,\mathcal C_0(\H)\big) \\
    &\overset{\eqref{eq}}{\hookrightarrow} \ell^1(\G)\,\hat{\otimes}\,\ell^1(\H)\,\hat{\otimes}\,\big(\H\ltimes_{\rm lt}\mathcal C_0(\H)\big) \\
    &\cong \ell^1(\H)\,\hat{\otimes}\,\ell^1\big(\G,\mathcal K(\ell^2(\H)\big)). 
\end{align*} The final isomorphism holds because of the Stone-von Neumann theorem: $\H\ltimes_{\rm lt}\mathcal C_0(\H)\cong \mathcal K(\ell^2(\H)\big)$. $\ell^1\big(\G,\mathcal K(\ell^2(\H)\big))$ is symmetric because of \cite[Theorem 1]{Ku}, hence $\ell^1(\H)\,\hat{\otimes}\,\ell^1\big(\G,\mathcal K(\ell^2(\H)\big))$ is symmetric because of \cite[Theorem 5]{LeN}. We conclude that $\ell^1_\Gamma\big(\G',\mathcal C_0(\H)\big)$ and thus $\ell^{\infty,1}(\Xi)$ are symmetric.
\end{proof}

\begin{cor}\label{pair}
The pair groupoid over a discrete set $\U$ is hypersymmetric.
\end{cor}

Now we will proceed to upgrade Theorem \ref{trans} to the case of general discrete groupoids.

\begin{thm}\label{full}
A discrete groupoid $\Xi$ is symmetric (resp. hypersymmetric) if and only if its isotropy subgroups are symmetric (resp. hypersymmetric). 
\end{thm}

\begin{proof}
Is obvious that symmetry (resp. hypersymmetry) of $\Xi$ implies the symmetry (resp. hypersymmetry) of its isotropy groups ($\ell^{1}(\Xi_u^u,\mathfrak A)\hookrightarrow \ell^{\infty,1}(\Xi,\mathfrak A)$ in a natural way). As the 'only if' part is clear, let us prove the 'if' part.

Let $\mathfrak A$ be a $C^*$-algebra and $\Phi$ be an arbitrary section in $\ell^{\infty,1}(\Xi,\mathfrak A)$. $\Phi$ has a countable support (see Remark \ref{convergence}), and since every discrete groupoid can be decomposed as a disjoint union of discrete transitive subgroupoids, we can find a countable number of disjoint transitive subgroupoids $\{\Xi(i)\}_{i\in \N}$, such that ${\rm supp}(\Phi)\subset \bigcup_{i=1}^\infty \Xi(i) $. Define $\Phi_n\in \ell^{\infty,1}(\Xi,\mathfrak A)$ as $$\Phi_n(x):=\left\{\begin{array}{ll}
\Phi(x)    & \textup{if\ } x\in \bigcup_{i=1}^n \Xi(i) , \\
0_\mathfrak A     & \textup{if\ } x\not\in \bigcup_{i=1}^n \Xi(i) . \\
\end{array}\right.$$ We have that $\lim_n\Phi_n= \Phi$ in $\ell^{\infty,1}(\Xi,\mathfrak A)$, but more importantly, 
\begin{equation}\label{spectra}
    {\rm Spec}_{\ell^{\infty,1}(\Xi,\mathfrak A)}(\Phi)=\bigcup_{i=1}^\infty{\rm Spec}_{\ell^{\infty,1}(\Xi,\mathfrak A)}(\Phi_n).
\end{equation} So, to conclude, its enough to prove that ${\rm Spec}_{\ell^{\infty,1}(\Xi,\mathfrak A)}(\Phi_n)\subset \R$, whenever $\Phi^*=\Phi$. Indeed, if $\Phi^*=\Phi$, after fixing $n$ we see that $\Phi_n^*=\Phi_n$ and $\ell^{\infty,1}(\Xi,\mathfrak A)$ contains an isometric $^*$-isomorphic copy of $\ell^{\infty,1}(\bigcup_{i=1}^n\Xi(i),\mathfrak A)$, obtained by extending the sections in the latter algebra with zeros outside its original domain. So by definition, one has that $\Phi_n\in\ell^{\infty,1}(\bigcup_{i=1}^n\Xi(i),\mathfrak A) \subset \ell^{\infty,1}(\Xi,\mathfrak A)$, which implies that
\begin{equation*}
    {\rm Spec}_{\ell^{\infty,1}(\Xi,\mathfrak A)}(\Phi_n)\subset{\rm Spec}_{\ell^{\infty,1}(\bigcup_{i=1}^n\Xi(i),\mathfrak A)}(\Phi_n).
\end{equation*}
But $\ell^{\infty,1}(\bigcup_{i=1}^n\Xi(i),\mathfrak A)\cong \bigoplus_{i=1}^n \ell^{\infty,1}(\Xi(i),\mathfrak A)$ and since the (finite) direct sum of symmetric Banach $^*$-algebras is symmetric, ${\rm Spec}_{\ell^{\infty,1}(\bigcup_{i=1}^n\Xi(i),\mathfrak A)}(\Phi_n)\subset \R$ and the result follows.
\end{proof}

\begin{ex}
Let $\Xi=\G\ltimes_\gamma\!X$ be a transformation groupoid (see Definition \ref{starlet}), where $X$ is discrete. In this case, $\mathcal U=X$ and $\Xi_u^u=\{g\in G\mid \gamma_g(u)=u\}\times\{u\}$, which is identificable with the stabilizer of $u$, namely ${\rm Stab}_\gamma(u)\leq \G$.
\end{ex}

\begin{rem}
We will finish this paper with a bit of wishful thinking: Is reasonable to believe that the groupoid analog of \cite[Theorem 3.3]{FJM} could be true. For a general locally compact groupoid one can define symmetry and hypersymmetry for $\Xi$ as the symmetry of the Hahn algebras $L^{\infty,1}(\Xi)$ and $L^{\infty,1}(\Xi\!\mid\!\mathscr C)$, respectively. So we may pose two problems: \begin{enumerate}
    \item[$(i)$] Is the (hyper)symmetry of a Hausdorff locally compact groupoid $\Xi$ implied by the (hyper)symmetry of its discretization $\Xi^{\rm dis}$?
    \item[$(ii)$] Is Corollary \ref{usecor} still valid for \'etale groupoids?
\end{enumerate} In view of Theorem \ref{promised}, $(i)$ is equivalent to \begin{enumerate}
    \item[$(i')$] Is the (hyper)symmetry of a Hausdorff locally compact groupoid $\Xi$ implied by the (hyper)symmetry of its discretized isotropy groups $(\Xi_u^u)^{\rm dis}$?
\end{enumerate}
\end{rem}


\bigskip
\bigskip
ADDRESS

\smallskip
F. Flores

Department of Mathematics, University of Virginia,

114 Kerchof Hall. 141 Cabell Dr,

Charlottesville, Virginia, United States

E-mail: hmy3tf@virginia.edu

\end{document}